\documentclass[12pt,oneside,english]{amsart}
\usepackage[T1]{fontenc}
\usepackage[latin9]{inputenc}
\usepackage{geometry}
\usepackage[active]{srcltx}
\usepackage{babel}
\usepackage{float}
\usepackage{amstext}
\usepackage{amsthm}
\usepackage{graphicx}
\usepackage{esint}
\usepackage[unicode=true,pdfusetitle,
 bookmarks=true,bookmarksnumbered=false,bookmarksopen=false,
 breaklinks=false,pdfborder={0 0 1},backref=false,colorlinks=false]
 {hyperref}

\makeatletter
\numberwithin{equation}{section}
\numberwithin{figure}{section}
\theoremstyle{plain}
\newtheorem{conjecture}{\protect\conjecturename}[section]
\theoremstyle{plain}
\newtheorem{thm}{\protect\theoremname}[section]
\theoremstyle{plain}
\newtheorem{lem}{\protect\lemmaname}[section]

\usepackage{babel}

\providecommand{\lemmaname}{Lemma}
\providecommand{\theoremname}{Theorem}

\providecommand{\conjecturename}{Conjecture}

\makeatother

\providecommand{\conjecturename}{Conjecture}
\providecommand{\lemmaname}{Lemma}
\providecommand{\theoremname}{Theorem}

\begin{document}
\title[On nonnegativity of $M_{C1}(m,n)$ and $M_{C5}(m,n)$]{Proof of a conjecture of Garvan and Jennings-Shaffer on the nonnegativity
of $M_{C1}(m,n)$ and $M_{C5}(m,n)$ }
\author{Bing He}
\address{School of Mathematics and Statistics, HNP-LAMA, Central South University
\\
 Changsha 410083, Hunan, People's Republic of China}
\email{yuhelingyun@foxmail.com; yuhe123456@foxmail.com}
\author{Shuming Liu}
\address{School of Mathematics and Statistics, HNP-LAMA, Central South University
\\
 Changsha 410083, Hunan, People's Republic of China}
\email{232111040@csu.edu.cn}
\thanks{The first author is the corresponding author.}
\begin{abstract}
In their 2016 paper on exotic Bailey--Slater SPT-functions, Garvan
and Jennings-Shaffer introduced many new spt-crank-type functions
and proposed a conjecture that the spt-crank-type functions $M_{C1}(m,n)$
and $M_{C5}(m,n)$ are both nonnegative for all $m\in\mathbb{Z}$
and $n\in\mathbb{N}.$ Applying Wright\textquoteright s circle method,
Jang and Kim showed that $M_{C1}(m,n)$ and $M_{C5}(m,n)$ are both
positive for a fixed integer $m$ and large enough integers $n.$
Up to now, no complete proof of this conjecture has been given. In
this paper, we provide a complete proof for this conjecture by using
the theory of lattice points. Our proof is quite different from that
of Jang and Kim. 
\end{abstract}

\maketitle

\section{Introduction}

A partition of a positive integer $n$ is a non-increasing sequence
of positive integers whose sum is $n.$ Let $p(n)$ be the number
of integer partitions of a positive integer $n$ and let $p(0):=1.$
Then the generating function for $p(n)$ is 
\[
\sum_{n\geq0}p(n)q^{n}=\frac{1}{(q;q)_{\infty}},
\]
where $(z;q)_{\infty}$ is the $q$-shifted factorial given by \cite{GR}
\[
(z;q)_{\infty}:=\prod_{j\geq0}(1-zq^{j}),\;|q|<1.
\]
The famous Ramanujan partition congruences state that 
\begin{align}
p(5n+4) & \equiv0(\bmod\:5),\label{eq:2-1}\\
p(7n+5) & \equiv0(\bmod\:7),\label{eq:2-2}\\
p(11n+6) & \equiv0(\bmod\:11).\label{eq:2-3}
\end{align}
To explain Ramanujan\textquoteright s partition congruences \eqref{eq:2-1},
\eqref{eq:2-2} and \eqref{eq:2-3}, Dyson \cite{D} introduced the
rank of a partition as the largest part minus the number of parts.
Furthermore, Dyson conjectured that the partitions of $5n+4$ classified
by their rank modulo 5 will produce five sets with equal cardinality
$p(5n+4)/5.$ He also conjectured that the partitions of $7n+5,$
classified by rank, will split into seven sets of equal cardinality
$p(7n+5)/7.$ Let $N(m,n)$ denote the number of partitions of $n$
whose rank is $m.$ Then a generating function for $N(m,n)$ is as
follows \cite{D}:
\[
R(z,q):=\sum_{n\geq0}\sum_{m=-\infty}^{\infty}N(m,n)z^{m}q^{n}=\sum_{n\geq0}\frac{q^{n^{2}}}{(zq;q)_{n}(z^{-1}q;q)_{n}},
\]
where
\[
(\zeta;q)_{0}:=1,\quad(\zeta;q)_{n}:=\prod_{k=1}^{n}(1-\zeta q^{k-1}),\;n\geq1.
\]
If $z=1,$ then the above formula reduces to Euler's celebrated generating
function for $p(n):$
\[
\sum_{n=0}^{\infty}p(n)q^{n}=\sum_{n=0}^{\infty}\frac{q^{n^{2}}}{(q;q)_{n}^{2}}.
\]

Let $N(r,m,n)$ denote the number of partitions of $n$ whose rank
is congruent to $r$ modulo $m.$ Later, these conjectures were proved
by A.O.L. Atkin and H.P.F. Swinnerton-Dyer \cite{AS} using the generating
functions of the rank difference $N(s,l,ln+d)-N(t,l,ln+d)$ with $l=5,7$
and $0\leq d,s,t<l.$ They proved that
\[
N(r,5,5n+4)=\frac{p(5n+4)}{5},\quad0\leq r\leq4,
\]
and
\[
N(r,7,7n+5)=\frac{p(7n+5)}{7},\quad0\leq r\leq6.
\]
However, the rank can not provide a purely combinatorial description
of the Ramanujan congruence \eqref{eq:2-3}. In his paper \cite{D},
Dyson also conjectured the existence of another partition statistics,
the cranks of partitions, which gives a combinatorial interpretation
of Ramanujan\textquoteright s congruence for the partition function
modulo 11. This partition statistics was later discovered by Andrews
and Garvan in their paper \cite{AG}. For $n>1,$ let $M(m,n)$ denote
the number of partitions of $n$ whose crank is $m$ while for $n\leq1$
we set 
\[
M(m,n)=\begin{cases}
-1, & \mathrm{if}\;(m,n)=(0,1),\\
1, & \mathrm{if}\;(m,n)=(0,0),(1,1),(-1,1),\\
0, & \mathrm{otherwise}.
\end{cases}
\]
The generating function for $M(m,n)$ is as follows: 
\[
C(z,q):=\sum_{m=-\infty}^{\infty}\sum_{n=0}^{\infty}M(m,n)z^{m}q^{n}=\frac{(q;q)_{\infty}}{(zq;q)_{\infty}(z^{-1}q;q)_{\infty}}.
\]
  Let $M(m,L;n)$ be the number of partitions of $n$ with crank
congruent to $m$ modulo $L.$ In \cite{AG}, Andrews and Garvan proved
that 
\[
M(m,11;11n+6)=\frac{p(11n+6)}{11},\quad0\leq m\leq10.
\]
A systematic account of ranks and cranks can be found in Chapters
2, 3, and 4 of Ramanujan\textquoteright s Lost Notebook \cite{AB}.

Let $\mathrm{spt}(n)$ denote the total number of appearances of the
smallest part in each integer partition of $n.$ A generating function
for $\mathrm{spt}(n)$ is as follows:
\[
S(q):=\sum_{n=1}^{\infty}\mathrm{spt}(n)q^{n}=\sum_{n=1}^{\infty}\frac{q^{n}}{\left(1-q^{n}\right)^{2}\left(q^{n+1};q\right)_{\infty}}.
\]
In \cite{A}, Andrews introduced the spt function $\mathrm{spt}(n)$
and proved three striking congruences:
\begin{align}
\mathrm{spt}(5n+4) & \equiv0(\bmod\:5),\label{eq:2-4}\\
\mathrm{spt}(7n+5) & \equiv0(\bmod\:7),\label{eq:2-5}\\
\mathrm{spt}(13n+6) & \equiv0(\bmod\:13).\nonumber 
\end{align}
These congruences are reminiscent of Ramanujan\textquoteright s partition
congruences \eqref{eq:2-1}, \eqref{eq:2-2} and \eqref{eq:2-3}.
In \cite{AGL}, Andrews, Garvan, and Liang introduced an spt-crank
$N_{S}(m,t,n)$ and proved that 
\[
N_{S}(k,5,5n+4)=\frac{\mathrm{spt}(5n+4)}{5}\text{ for }0\leq k\leq4,
\]
and 
\[
N_{S}(k,7,7n+5)=\frac{\mathrm{spt}(7n+5)}{7}\text{ for }0\leq k\leq6,
\]
which provide new combinatorial interpretations of the congruences
\eqref{eq:2-4} and \eqref{eq:2-5} of the spt function. They defined
a two variable generalization of the spt function: 
\[
S(z,q):=\sum_{n=1}^{\infty}\sum_{m=-\infty}^{\infty}N_{S}(m,n)z^{m}q^{n}=\sum_{n=1}^{\infty}\frac{q^{n}\left(q^{n+1};q\right)_{\infty}}{\left(zq^{n};q\right)_{\infty}\left(z^{-1}q^{n};q\right)_{\infty}}.
\]
It is obvious that $S(q)=S(1,q).$ Introducing an extra variable,
they obtained a new statistic and refinement of the original counting
function. In the same paper, by examining $S(\zeta_{5},q)$ and $S(\zeta_{7},q)$
with $\zeta_{5}=\mathrm{e}^{2\pi\mathrm{i}/5}$ and $\zeta_{7}=\mathrm{e}^{2\pi\mathrm{i}/7},$
Andrews, Garvan, and Liang reproved the congruences \eqref{eq:2-4}
and \eqref{eq:2-5}. Applying Bailey\textquoteright s lemma to a Bailey
pair, they deduced an interesting formula between $R(z,q),C(z,q)$
and $S(z,q):$
\begin{equation}
(1-z)\left(1-z^{-1}\right)S(z,q)=R(z,q)-C(z,q).\label{eq:2-6}
\end{equation}

An overpartition is a partition in which the first occurrence of a
part may be overlined. In 2014, Garvan and Jennings-Shaffer \cite{GJ2}
applied Bailey\textquoteright s Lemma to four different Bailey pairs
to obtain several new spt-crank functions and proved many congruences
on three spt functions for overpartitions and the spt function for
partitions with smallest part even and without repeated odd parts.
Let $\overline{\mathrm{spt}}(n)$ denote the total number of appearances
of the smallest parts among the overpartitions of n whose smallest
part is not overlined and let $\mathrm{M2spt}(n)$ denote the number
of smallest parts in the partitions of $n$ without repeated odd parts
and smallest part even. The generating functions for $\overline{\mathrm{spt}}(n)$
and $\mathrm{M2spt}(n)$ are as follows:
\begin{align*}
\sum_{n=1}^{\infty}\overline{\mathrm{spt}}(n)q^{n} & =\sum_{n=1}^{\infty}\frac{q^{n}\left(-q^{n+1};q\right)_{\infty}}{\left(1-q^{n}\right)^{2}\left(q^{n+1};q\right)_{\infty}},\\
\sum_{n=1}^{\infty}\mathrm{M2spt}(n)q^{n} & =\sum_{n=1}^{\infty}\frac{q^{2n}\left(-q^{2n+1};q^{2}\right)_{\infty}}{\left(1-q^{2n}\right)^{2}\left(q^{2n+2};q^{2}\right)_{\infty}}.
\end{align*}
In the same paper, they defined two variable generalizations for $\overline{\mathrm{spt}}(n)$
and $\mathrm{M2spt}(n)$ by
\begin{align*}
\overline{\mathrm{S}}(z,q) & :=\sum_{n=1}^{\infty}\frac{q^{n}\left(-q^{n+1},q^{n+1};q\right)_{\infty}}{\left(zq^{n};q\right)_{\infty}\left(z^{-1}q^{n};q\right)_{\infty}},\\
\mathrm{S}2(z,q) & :=\sum_{n=1}^{\infty}\frac{q^{2n}\left(-q^{2n+1},q^{2n+2};q^{2}\right)_{\infty}}{\left(zq^{2n};q^{2}\right)_{\infty}\left(z^{-1}q^{2n};q^{2}\right)_{\infty}}.
\end{align*}
Let $\overline{R}(z,q)$ denote the generating function for the Dyson
rank\footnote{The Dyson rank of an overpartition is the largest part minus the number
of parts.} of an overpartition and let $R2(z,q)$ denote the $M_{2}$-rank\footnote{The $M_{2}$-rank of a partition without repeated odd parts is the
ceiling of half the largest part minus the number of parts.} of a partition without repeated odd parts. Applying Bailey\textquoteright s
lemma to two different Bailey pair, they \cite{GJ2} found two beautiful
identities:
\begin{align*}
(1-z)\left(1-z^{-1}\right)\overline{\mathrm{S}}(z,q) & =\overline{R}(z,q)-(-q;q)_{\infty}C(z,q),\\
(1-z)\left(1-z^{-1}\right)\mathrm{S}2(z,q) & =R2(z,q)-\left(-q;q^{2}\right)_{\infty}C\left(z,q^{2}\right).
\end{align*}
These two formulas are very similar to \eqref{eq:2-6}.

Motivated by Andrews--Garvan--Liang \cite{AGL} and Garvan--Jennings-Shaffer
\cite{GJ2}, Garvan and Jennings-Shaffer \cite{GJ} studied a family
of spt-crank-type functions with the form: 
\[
\frac{P(q)}{\left(z;q\right)_{\infty}\left(z^{-1};q\right)_{\infty}}\sum_{n=1}^{\infty}\left(z,z^{-1};q\right)_{n}q^{n}\beta_{n},
\]
where $P(q)$ is an infinite product and $\beta_{n}$ orginates from
a Bailey pair relative to $(1,q)$ and considered the Bailey pairs
from group A in \cite{S1} as well as the Bailey pairs $C(1),C(5),E(2)$
and $E(4).$ Given a Bailey pair $(\alpha^{X},\beta^{X}),$ they defined
the corresponding two variable spt-crank-type function by 
\[
S_{X}(z,q):=\frac{P_{X}(q)}{\left(z;q\right)_{\infty}\left(z^{-1};q\right)_{\infty}}\sum_{n=1}^{\infty}\left(z,z^{-1};q\right)_{n}\beta_{n}^{X}q^{n}=:\sum_{n=1}^{\infty}\sum_{m=-\infty}^{\infty}M_{X}(m,n)z^{m}q^{n}.
\]
In particular, 
\begin{align*}
\sum_{n=1}^{\infty}\sum_{m=-\infty}^{\infty}M_{A1}(m,n)z^{m}q^{n} & =\frac{(q;q)_{\infty}}{\left(z;q\right)_{\infty}\left(z^{-1};q\right)_{\infty}}\sum_{n=1}^{\infty}\frac{q^{n}\left(z;q\right)_{n}\left(z^{-1};q\right)_{n}}{(q;q)_{2n}}\\
 & =\sum_{n=1}^{\infty}\frac{q^{n}\left(q^{2n+1};q\right)_{\infty}}{\left(zq^{n},z^{-1}q^{n};q\right)_{\infty}},
\end{align*}
\begin{align*}
\sum_{n=1}^{\infty}\sum_{m=-\infty}^{\infty}M_{A3}(m,n)z^{m}q^{n} & =\frac{(q;q)_{\infty}}{\left(z;q\right)_{\infty}\left(z^{-1};q\right)_{\infty}}\sum_{n=1}^{\infty}\frac{q^{2n}\left(z;q\right)_{n}\left(z^{-1};q\right)_{n}}{(q;q)_{2n}}\\
 & =\sum_{n=1}^{\infty}\frac{q^{2n}\left(q^{2n+1};q\right)_{\infty}}{\left(zq^{n},z^{-1}q^{n};q\right)_{\infty}},
\end{align*}
\begin{align*}
\sum_{n=1}^{\infty}\sum_{m=-\infty}^{\infty}M_{A5}(m,n)z^{m}q^{n} & =\frac{(q;q)_{\infty}}{\left(z;q\right)_{\infty}\left(z^{-1};q\right)_{\infty}}\sum_{n=1}^{\infty}\frac{q^{n^{2}+n}\left(z;q\right)_{n}\left(z^{-1};q\right)_{n}}{(q;q)_{2n}}\\
 & =\sum_{n=1}^{\infty}\frac{q^{n^{2}+n}\left(q^{2n+1};q\right)_{\infty}}{\left(zq^{n},z^{-1}q^{n};q\right)_{\infty}},
\end{align*}
\begin{align*}
\sum_{n=1}^{\infty}\sum_{m=-\infty}^{\infty}M_{A7}(m,n)z^{m}q^{n} & =\frac{(q;q)_{\infty}}{\left(z;q\right)_{\infty}\left(z^{-1};q\right)_{\infty}}\sum_{n=1}^{\infty}\frac{q^{n^{2}}\left(z;q\right)_{n}\left(z^{-1};q\right)_{n}}{(q;q)_{2n}}\\
 & =\sum_{n=1}^{\infty}\frac{q^{n^{2}}\left(q^{2n+1};q\right)_{\infty}}{\left(zq^{n},z^{-1}q^{n};q\right)_{\infty}},
\end{align*}
\begin{align*}
\sum_{n=1}^{\infty}\sum_{m=-\infty}^{\infty}M_{C1}(m,n)z^{m}q^{n} & =\frac{\left(q;q^{2}\right)_{\infty}(q;q)_{\infty}}{\left(z;q\right)_{\infty}\left(z^{-1};q\right)_{\infty}}\sum_{n=1}^{\infty}\frac{q^{n}\left(z;q\right)_{n}\left(z^{-1};q\right)_{n}}{\left(q;q^{2}\right)_{n}(q;q)_{n}}\\
 & =\sum_{n=1}^{\infty}\frac{q^{n}\left(q^{2n+1};q^{2}\right)_{\infty}\left(q^{n+1};q\right)_{\infty}}{\left(zq^{n},z^{-1}q^{n};q\right)_{\infty}},
\end{align*}
\begin{align*}
\sum_{n=1}^{\infty}\sum_{m=-\infty}^{\infty}M_{C5}(m,n)z^{m}q^{n} & =\frac{\left(q;q^{2}\right)_{\infty}(q;q)_{\infty}}{\left(z;q\right)_{\infty}\left(z^{-1};q\right)_{\infty}}\sum_{n=1}^{\infty}\frac{q^{\frac{n^{2}+n}{2}}\left(z;q\right)_{n}\left(z^{-1};q\right)_{n}}{\left(q;q^{2}\right)_{n}(q;q)_{n}}\\
 & =\sum_{n=1}^{\infty}\frac{q^{\frac{n^{2}+n}{2}}\left(q^{2n+1};q^{2}\right)_{\infty}\left(q^{n+1};q\right)_{\infty}}{\left(zq^{n},z^{-1}q^{n};q\right)_{\infty}},
\end{align*}
\begin{align*}
\sum_{n=1}^{\infty}\sum_{m=-\infty}^{\infty}M_{E2}(m,n)z^{m}q^{n} & =\frac{\left(q^{2};q^{2}\right)_{\infty}}{\left(z;q\right)_{\infty}\left(z^{-1};q\right)_{\infty}}\sum_{n=1}^{\infty}\frac{(-1)^{n}q^{n}\left(z;q\right)_{n}\left(z^{-1};q\right)_{n}}{\left(q^{2};q^{2}\right)_{n}}\\
 & =\sum_{n=1}^{\infty}\frac{(-1)^{n}q^{n}\left(q^{2n+2};q^{2}\right)_{\infty}}{\left(zq^{n},z^{-1}q^{n};q\right)_{\infty}},
\end{align*}
and
\begin{align*}
\sum_{n=1}^{\infty}\sum_{m=-\infty}^{\infty}M_{E4}(m,n)z^{m}q^{n} & =\frac{\left(q^{2};q^{2}\right)_{\infty}}{\left(z;q\right)_{\infty}\left(z^{-1};q\right)_{\infty}}\sum_{n=1}^{\infty}\frac{q^{2n}\left(z;q\right)_{n}\left(z^{-1};q\right)_{n}}{\left(q^{2};q^{2}\right)_{n}}\\
 & =\sum_{n=1}^{\infty}\frac{q^{2n}\left(q^{2n+2};q^{2}\right)_{\infty}}{\left(zq^{n},z^{-1}q^{n};q\right)_{\infty}}.
\end{align*}
They proved various congruences for these functions by using dissection
formulas for $S_{X}(\zeta,q)$ and found single series and product
identities for $S_{A5}(z,q),S_{A7}(z,q),$ $S_{C5}(z,q)$ and $S_{E2}(z,q)$
as well as Hecke-Rogers type double sum formulas for $S_{A1}(z,q),$
$S_{A3}(z,q),S_{C1}(z,q)$ and $S_{E4}(z,q).$ At the end of their
paper, they proposed the following conjecture suggested by numerical
calculations.
\begin{conjecture}
\label{conj1.1} For $m\in\mathbb{Z}$ and $n\in\mathbb{N},$ $M_{C1}(m,n)$
and $M_{C5}(m,n)$ are both nonnegative\footnote{In this paper, we let $\mathbb{Z}$ and $\mathbb{N}$ denote the set
of integers and the set of positive integers respectively.}.
\end{conjecture}
In \cite{JK}, using Wright\textquoteright s circle method \cite{W},
Jang and Kim\footnote{In \cite{JK}, Jang and Kim write $M_{C1}(m,n)$ and $M_{C5}(m,n)$
as $N_{C1}(m,n)$ and $N_{C5}(m,n)$ respectively.} derived asymptotic formulas for $M_{C1}(m,n)$ and $M_{C5}(m,n)$
and showed that $M_{C1}(m,n)$ and $M_{C5}(m,n)$ are both positive
for a fixed integer $m$ and large enough integers $n.$

As far as we are concerned, no complete proof of Conjecture \ref{conj1.1}
has been given so far. In this paper, we shall supply a complete proof
of this conjecture by applying the theory of lattice points. Our method
is quite different from that of Jang and Kim \cite{JK}.

Let us now briefly outline the key steps to prove Conjecture \ref{conj1.1}.

The nonnegativity of $M_{C1}(m,n)$ is more complicated than that
of $M_{C5}(m,n).$ We first consider the nonnegativity of $M_{C1}(m,n).$

In order to study the nonnegativity of $M_{C1}(m,n),$ we need to
consider the nonnegativity of $X^{(m)}(n),$ where $X^{(m)}(n)$ is
defined by 
\begin{equation}
\sum_{n\geq0}X^{(m)}(n)q^{n}:=\frac{1}{1-q^{2}}\left(\sum_{n=1}^{\infty}\frac{(-1)^{n}q^{n(3n+1)+2mn}}{1-q^{2n}}-\sum_{n=1}^{\infty}\frac{(-1)^{n}q^{\frac{n(n+1)}{2}+mn}}{1-q^{n}}\right).\label{eq:1-13}
\end{equation}
When $n$ is even, we can relate the coefficient $X^{(m)}(n)$ to
the numbers of lattice points in some region of $\mathbb{R}^{2}.$
With the help of an important theorem \cite[Theorem 9.2, p.123]{H}
due to the Czech mathematician M.V. Jarnik, we can get an estimate
for the lower bound of $X^{(m)}(n).$ This, together with other results,
allows us to obtain the nonnegativity of $X^{(m)}(n).$

The first result is as follows.
\begin{thm}
\label{L2} Let $m$ and $n$ be two non-negative integers and let
\begin{equation}
f(m):=\left(\frac{2\left(6+\sqrt{36+(m+2)\log2}\right)}{\log2}\right)^{2}.\label{eq:t1-1}
\end{equation}
 If $n$ is odd, then $X^{(m)}(n)$ is nonnegative; if $n$ is even,
then $X^{(m)}(n)>0$ for $n\geq f(m).$
\end{thm}
In addition to this result, we also consider the case $0\leq n\leq20m.$
\begin{thm}
\label{lemma2.8} Let $m$ and $n$ be two non-negative integers such
that $n\leq20m.$ Then $X^{(m)}(n)$ is nonnegative.
\end{thm}
We next consider the nonnegativity of $M_{C5}(m,n).$ In order to
prove the nonnegativity of $M_{C5}(m,n),$ we first study the nonnegativity
of $Y^{(m)}(n),$ where $Y^{(m)}(n)$ is defined by

\begin{equation}
\sum_{n\geq0}Y^{(m)}(n)q^{n}:=\sum_{n=1}^{\infty}\frac{(-1)^{n}q^{n(n+1)+2mn}}{1-q^{2n}}-\sum_{n=1}^{\infty}\frac{(-1)^{n}q^{\frac{n(n+1)}{2}+mn}}{1-q^{n}}.\label{eq:1-1}
\end{equation}
We convert the nonnegativity of $Y^{(m)}(n)$ to the consideration
of the numbers of four sets in $\mathbb{N}^{2}.$ From this we can
obtain the nonnegativity of $Y^{(m)}(n).$
\begin{thm}
\label{L1} Let $m$ be a non-negative integer. Then $Y^{(m)}(n)$
is nonnegative for all $n\geq0.$
\end{thm}
Combining Theorems \ref{L2}, \ref{lemma2.8} and \ref{L1}, we can
confirm Conjecture \ref{conj1.1}\footnote{In \cite{CHM}, Chan, Ho and Mao observed an interesting phenomenon
that fewer terms in the denominator of some theta series are needed
and the corresponding theta series still possesses nonnegative coefficients.
See also Zhou's paper \cite{Z}. This observation can also apply to
our proof of Conjecture \ref{conj1.1}. See the proof of Conjecture
\ref{conj1.1} in Section \ref{sec:Pr} for more details.}.
\begin{thm}
\label{t} Conjecture \ref{conj1.1} is true.
\end{thm}
In the next section, we provide some auxiliary results, which are
crucial to the derivation of Theorems \ref{L2}, \ref{lemma2.8} and
\ref{L1}. Section \ref{sec:P} is devoted to our proofs of Theorems
\ref{L2}, \ref{lemma2.8} and \ref{L1}. In Section \ref{sec:Pr},
we show Theorem \ref{t}.

\section{Auxiliary results}

We begin this section by recalling the following result of Jarnik
concerning the problem of the number of lattice points inside a closed
curve.
\begin{lem}
\label{lemma2.2} (Jarnik, see \cite[Theorem 9.2, p.123]{H}) Let
$l\geq1$ be the length of a rectifiable simple closed curve and let
$A$ be the area of the region bounded by the curve. If $N$ is the
number of lattice points inside the curve, then 
\[
|N-A|<l.
\]
\end{lem}
The next result concerns the number of lattice points with $y$-coordinates
being odd.
\begin{lem}
\label{lemma2.3} Let $N$ represent the number of lattice points
in a bounded region $B,$ and let $M$ denote the number of lattice
points within this region where the $y$-coordinates are odd. Then
\[
\left|\frac{N}{2}-M\right|\leq\sup_{P,Q\in B}|x(P)-x(Q)|+1,
\]
where $x(T)$ denotes the $x$-coordinate of $T.$
\end{lem}
\begin{proof}
Let
\[
m_{1}:=\min\left\{ x|(x,y)\in B\;and\;(x,y)\in\mathbb{Z}^{2}\right\} 
\]
and 
\[
m_{2}:=\max\left\{ x|(x,y)\in B\;and\;(x,y)\in\mathbb{Z}^{2}\right\} .
\]
For each $x_{0}\in[m_{1},m_{2}]$ and $x_{0}\in\mathbb{Z},$ let $N_{x_{0}}$
be the number of lattice points in $B\cap\{(x_{0},y)|y\in\mathbb{R}\}$
and let $M_{x_{0}}$ be the number of lattice points in $B\cap\{(x_{0},y)|y\in\mathbb{R}\}$
with the $y$-coordinates being odd. Then $N_{x_{0}}-M_{x_{0}}$ is
the number of lattice points in $B\cap\{(x_{0},y)|y\in\mathbb{R}\}$
with the $y$-coordinates being even and so 
\[
\left|\left(N_{x_{0}}-M_{x_{0}}\right)-M_{x_{0}}\right|\leq1.
\]
It is easy to see that $N-M$ is the number of lattice points in the
region $B$ with the $y$-coordinates being even. Thus,
\begin{align*}
\left|\left(N-M\right)-M\right| & =\left|\sum_{\substack{x_{0}\in[m_{1},m_{2}]\\
x_{0}\in\mathbb{Z}
}
}\left(\left(N_{x_{0}}-M_{x_{0}}\right)-M_{x_{0}}\right)\right|\\
 & \leq\sum_{\substack{x_{0}\in[m_{1},m_{2}]\\
x_{0}\in\mathbb{Z}
}
}\left|\left(N_{x_{0}}-M_{x_{0}}\right)-M_{x_{0}}\right|\\
 & \leq m_{2}+1-m_{1}\\
 & \leq\sup_{P,Q\in B}|x(P)-x(Q)|+1.
\end{align*}
This implies that
\[
\left|\frac{N}{2}-M\right|\leq\sup_{P,Q\in B}|x(P)-x(Q)|+1,
\]
which finishes the proof.
\end{proof}
From Lemmas \ref{lemma2.2} and \ref{lemma2.3} we can derive the
following results.
\begin{lem}
\label{Lemma2.4} Let $m$ and $n$ be two non-negative integers and
\[
\Omega:=\left\{ (x,y)\bigg|xy<\frac{n+1}{2},y-6x<2m,y-4x>2m,x>0,y>0\right\} 
\]
and let $M_{1}^{(m)}(n)$ denote the number of lattice points in $\Omega$
with odd $y$-coordinates. Then
\begin{align*}
M_{1}^{(m)}(n) & <\frac{1}{4}(n+1)\left(\log\left(\frac{3\left(\sqrt{1+2t}-1\right)}{2\left(\sqrt{1+3t}-1\right)}\right)+\left(\frac{1}{\sqrt{1+3t}+1}+\frac{-1}{\sqrt{1+2t}+1}\right)\right)\\
 & \quad+2.2\sqrt{n+1}+1,
\end{align*}
where $t$ is defined by\footnote{The value of $t$ is set to be $+\infty$ if $m=0.$}
\begin{equation}
t=\frac{n+1}{m^{2}}.\label{eq:t}
\end{equation}
\end{lem}
\begin{figure}[H]
\caption{Schematic diagram of $\Omega$}
\end{figure}
\includegraphics[scale=0.15]{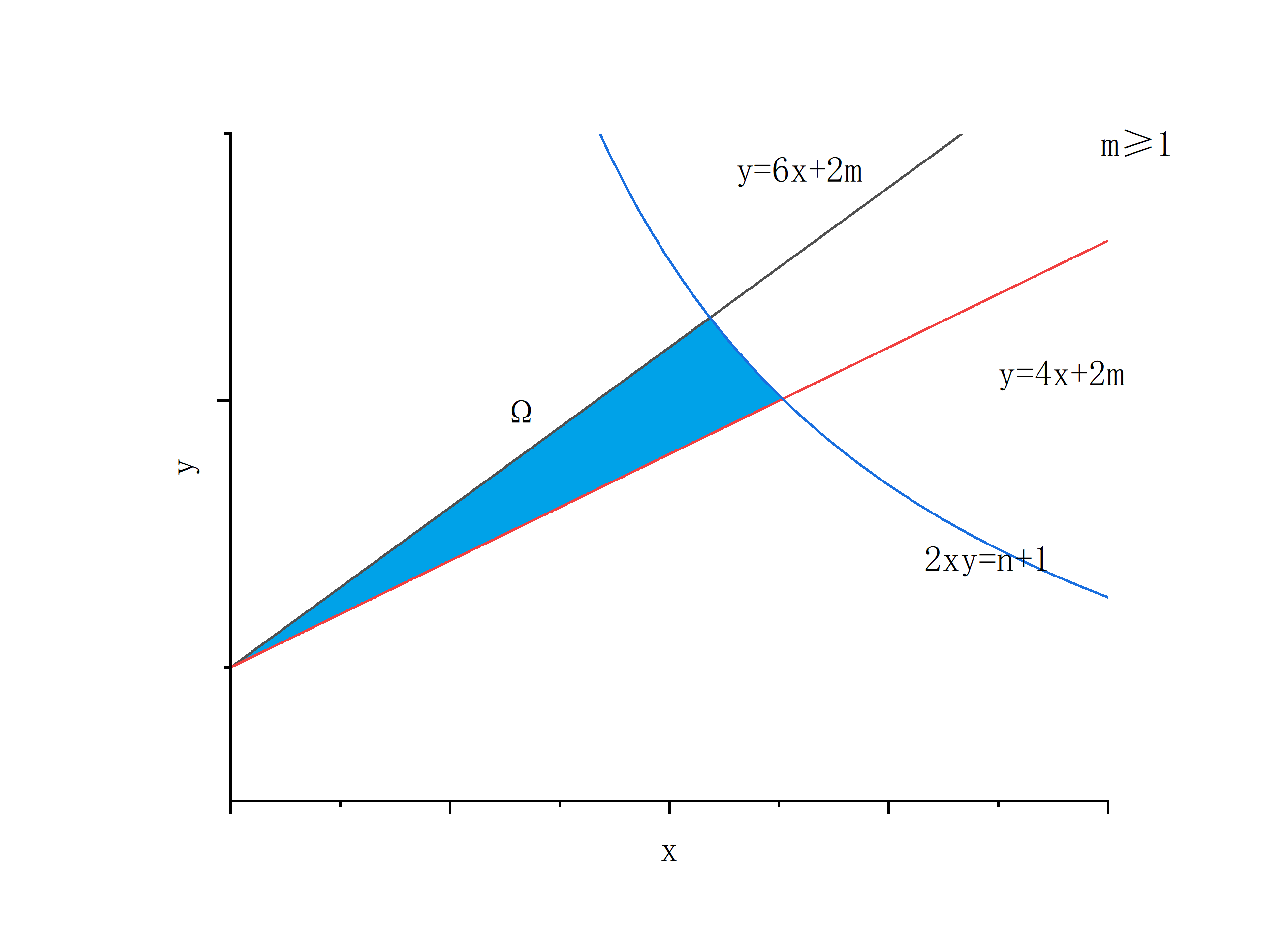} \includegraphics[scale=0.15]{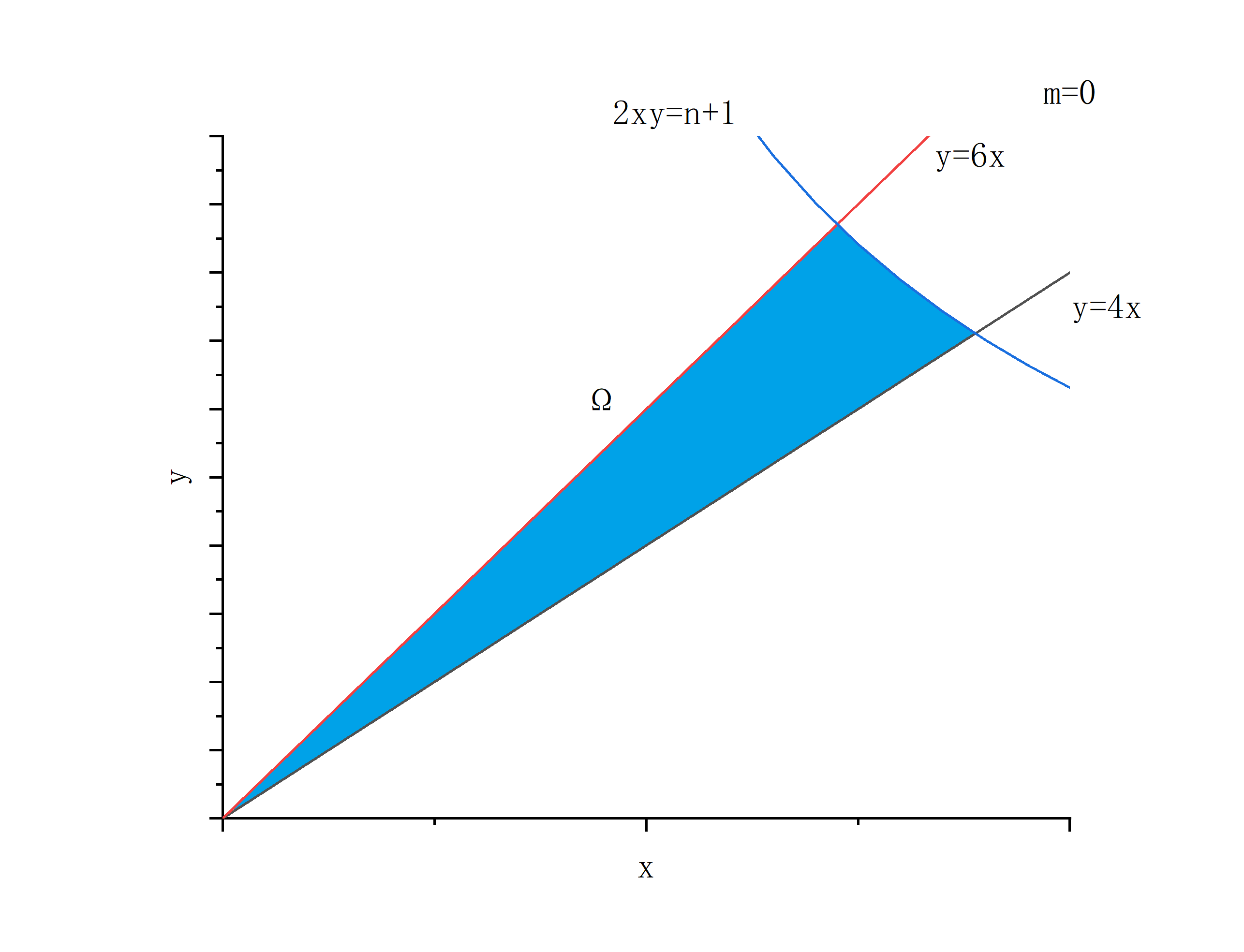}
\begin{proof}
Let $A(\Omega)$ and $l(\Omega)$ denote the area of $\Omega$ and
the length of the boundary curve of $\Omega$ respectively. Set 
\begin{align*}
 & x_{1}=0,\quad y_{1}=2m,\\
 & x_{2}=\frac{-2m+\sqrt{(2m)^{2}+8(n+1)}}{8},\quad y_{2}=\frac{n+1}{2x_{2}},\\
 & x_{3}=\frac{1}{12}\left(\sqrt{(2m)^{2}+12(n+1)}-(2m)\right),\quad y_{3}=\frac{n+1}{2x_{3}}.
\end{align*}
Then
\begin{equation}
\begin{aligned}A(\Omega) & =\int_{y_{2}}^{y_{3}}\frac{n+1}{2y}dy+\frac{(y_{2}-y_{1})x_{2}}{2}-\frac{(y_{3}-y_{1})x_{3}}{2}\\
 & =\frac{1}{2}(n+1)\log\left(\frac{3\left(\sqrt{(2m)^{2}+8(n+1)}-2m\right)}{2\left(\sqrt{(2m)^{2}+12(n+1)}-2m\right)}\right)\\
 & \;-\frac{m}{24}\left(3\sqrt{(2m)^{2}+8(n+1)}-2\sqrt{(2m)^{2}+12(n+1)}-m\right)\\
 & =\frac{1}{2}(n+1)\log\left(\frac{3\left(\sqrt{1+2t}-1\right)}{2\left(\sqrt{1+3t}-1\right)}\right)\\
 & \qquad+\frac{1}{2}(n+1)\left(\frac{1}{\sqrt{1+3t}+1}+\frac{-1}{\sqrt{1+2t}+1}\right)
\end{aligned}
\label{eq:1-4}
\end{equation}
and
\begin{equation}
\begin{aligned}l(\Omega) & =\sqrt{(x_{3}-x_{1})^{2}+(y_{3}-y_{1})^{2}}+\sqrt{(x_{2}-x_{1})^{2}+(y_{2}-y_{1})^{2}}\\
 & \qquad+\int_{y_{2}}^{y_{3}}\sqrt{1+\frac{(n+1)^{2}}{4y^{4}}}dy\\
 & <\sqrt{(x_{3}-x_{1})^{2}+(y_{3}-y_{1})^{2}}+\sqrt{(x_{2}-x_{1})^{2}+(y_{2}-y_{1})^{2}}\\
 & \qquad+(y_{3}-y_{2})\sqrt{1+\frac{(n+1)^{2}}{4y_{2}^{4}}}\\
 & =\frac{m}{2}\sqrt{\frac{17}{2}}\sqrt{-\sqrt{1+2t}+t+1}+\frac{m}{6}\sqrt{37}\sqrt{-2\sqrt{1+3t}+3t+2}\\
 & \qquad+m\left(\sqrt{1+3t}-\sqrt{1+2t}\right)\sqrt{\frac{\left(\sqrt{1+2t}-1\right)^{4}}{64t^{2}}+1}\\
 & =\frac{m}{2}\sqrt{\frac{17}{2}}\sqrt{1-\sqrt{1+2t}+t}+\frac{m}{6}\sqrt{37}\sqrt{2-2\sqrt{1+3t}+3t}\\
 & \qquad+\frac{mt}{\sqrt{1+3t}+\sqrt{1+2t}}\sqrt{\frac{t^{2}}{4\left(\sqrt{1+2t}+1\right)^{4}}+1}\\
 & <\frac{m}{2}\sqrt{\frac{17t}{2}}+\frac{m}{6}\sqrt{111t}+\frac{m\sqrt{t}}{\sqrt{3}+\sqrt{2}}\sqrt{\frac{1}{16}+1}\\
 & =\sqrt{\frac{17}{8}}\sqrt{n+1}+\sqrt{\frac{37}{12}}\sqrt{n+1}+\sqrt{\frac{17}{16}}\frac{1}{\sqrt{3}+\sqrt{2}}\sqrt{n+1}\\
 & <3.6\sqrt{n+1}.
\end{aligned}
\label{eq:1-5}
\end{equation}
It is easily seen that 
\begin{equation}
\begin{aligned}\sup_{P,Q\in\Omega}|x(P)-x(Q)| & =x_{2}-x_{1}=\frac{1}{8}\left(\sqrt{(2m)^{2}+8(n+1)}-(2m)\right)\\
 & =\frac{n+1}{\sqrt{(2m)^{2}+8(n+1)}+(2m)}\\
 & <\frac{\sqrt{2(n+1)}}{4}.
\end{aligned}
\label{eq:1-6}
\end{equation}
Let $N(\Omega)$ denote the number of lattice points in $\Omega$.
Applying Lemmas \ref{lemma2.2} and \ref{lemma2.3}, we get
\[
M_{1}^{(m)}(n)\leq\frac{N(\Omega)}{2}+\left|M_{1}^{(m)}(n)-\frac{N(\Omega)}{2}\right|
\]
\begin{align*}
 & \leq\frac{A(\Omega)+\left|N(\Omega)-A(\Omega)\right|}{2}+\left|M_{1}^{(m)}(n)-\frac{N(\Omega)}{2}\right|\\
 & <\frac{A(\Omega)}{2}+\frac{l(\Omega)}{2}+\sup_{P,Q\in\Omega}|x(P)-x(Q)|+1.
\end{align*}
Then the result follows readily by substituting \eqref{eq:1-4}, \eqref{eq:1-5}
and \eqref{eq:1-6} into the right side of the above inequality.
\end{proof}
\begin{lem}
\label{Lemma2.5} Let $m$ and $n$ be two non-negative integers and
\[
\Omega':=\left\{ (x,y)\bigg|xy<\frac{n+1}{2},2x-3y<2m,4x-y>2m,x>0,y>0\right\} 
\]
and let $M_{2}^{(m)}(n)$ denote the number of lattice points in $\Omega'$
with odd $y$-coordinates. Then 
\begin{align*}
M_{2}^{(m)}(n) & >\frac{1}{4}(n+1)\left(\frac{1}{\sqrt{1+3t}+1}+\frac{-1}{\sqrt{1+2t}+1}+\log\left(\frac{2\left(\sqrt{1+3t}+1\right)}{\sqrt{1+2t}+1}\right)\right)\\
 & \qquad-3.7\sqrt{n+1}-m-1,
\end{align*}
where $t$ is defined in \eqref{eq:t}.
\end{lem}
\begin{figure}[H]
\caption{Schematic diagram of $\Omega'$}
\end{figure}
 \includegraphics[scale=0.13]{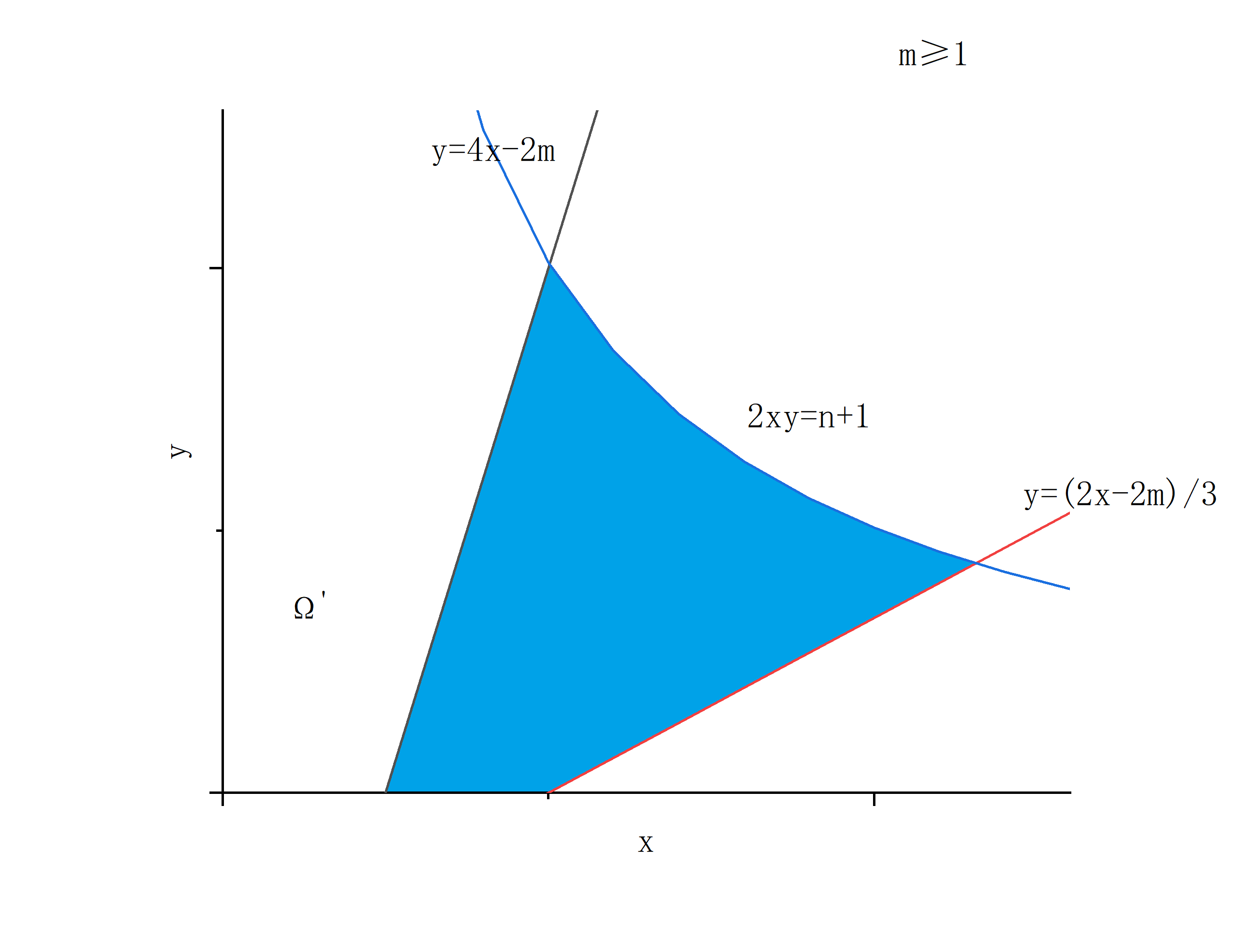} \includegraphics[scale=0.13]{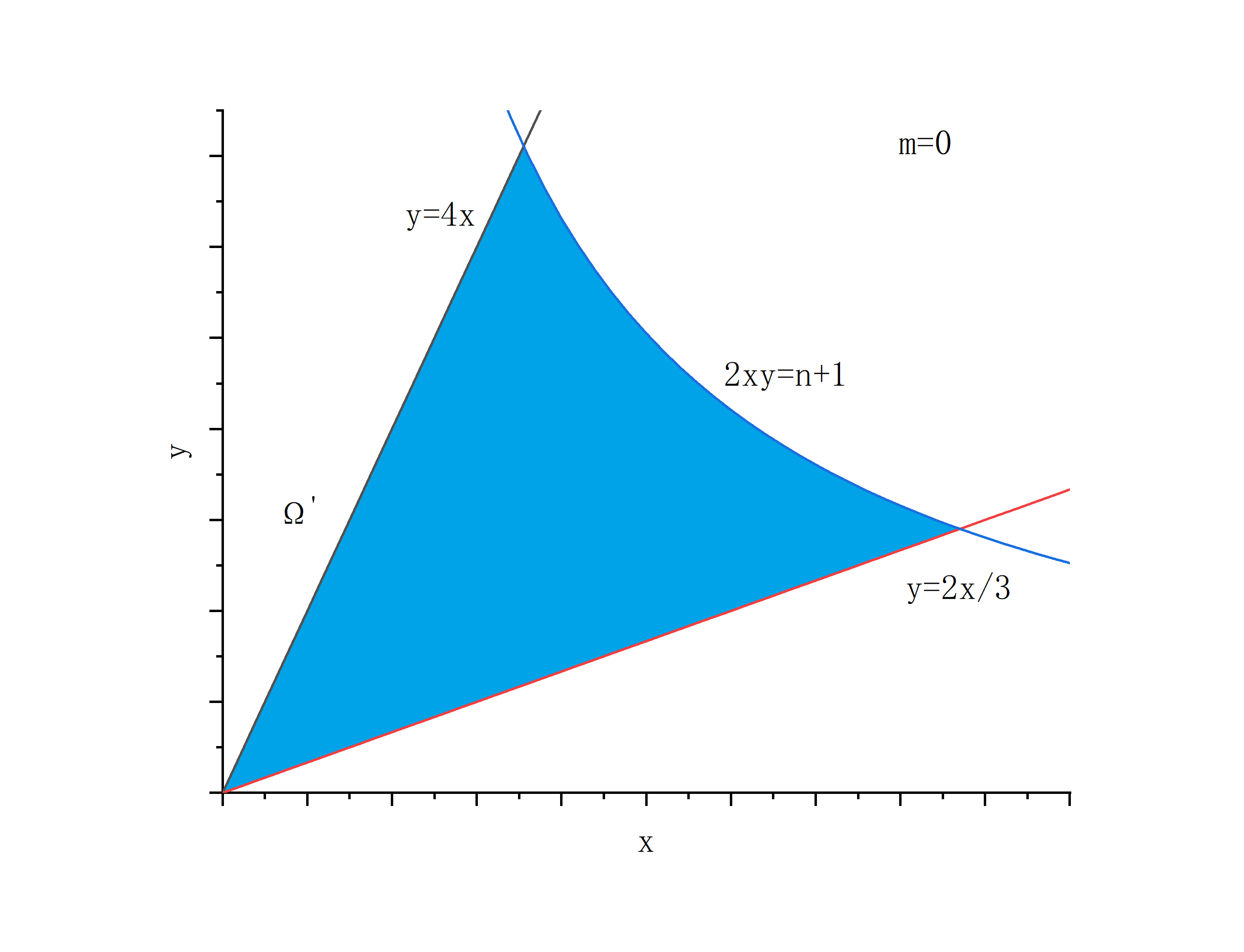}
\begin{proof}
Let $A(\Omega')$ and $l(\Omega')$ denote the area of $\Omega'$
and the length of the boundary curve of $\Omega'$ respectively. Set
\begin{align*}
 & x_{4}=\frac{m}{2},\quad y_{4}=0,\\
 & x_{5}=m,\quad y_{5}=0,\\
 & x_{6}=\frac{1}{8}\left(\sqrt{(2m)^{2}+8(n+1)}+2m\right),\quad y_{6}=\frac{n+1}{2x_{6}},\\
 & x_{7}=\frac{1}{4}\left(\sqrt{(2m)^{2}+12(n+1)}+2m\right),\quad y_{7}=\frac{n+1}{2x_{7}}.
\end{align*}
Then
\begin{equation}
\begin{aligned}A(\Omega') & =\frac{(x_{6}-x_{4})y_{6}}{2}+\int_{x_{6}}^{x_{7}}\frac{n+1}{2x}dx-\frac{(x_{7}-x_{5})y_{7}}{2}\\
 & =\frac{1}{2}(n+1)\left(\frac{2m}{\sqrt{(2m)^{2}+12(n+1)}+2m}+\frac{-2m}{\sqrt{(2m)^{2}+8(n+1)}+2m}\right.\\
 & \qquad\left.+\log2\left(\frac{\sqrt{(2m)^{2}+12(n+1)}+2m}{\sqrt{(2m)^{2}+8(n+1)}+2m}\right)\right)\\
 & =\frac{1}{2}(n+1)\left(\frac{1}{\sqrt{1+3t}+1}+\frac{-1}{\sqrt{1+2t}+1}+\log\left(\frac{2\left(\sqrt{1+3t}+1\right)}{\sqrt{1+2t}+1}\right)\right)
\end{aligned}
\label{eq:1-7}
\end{equation}
and 
\begin{align*}
l(\Omega') & =\sqrt{(x_{6}-x_{4})^{2}+(y_{6}-y_{4})^{2}}+\sqrt{(x_{7}-x_{5})^{2}+(y_{7}-y_{5})^{2}}\\
 & \qquad+x_{5}-x_{4}+\int_{x_{6}}^{x_{7}}\sqrt{1+\frac{(n+1)^{2}}{4x^{4}}}dx\\
 & <\sqrt{(x_{6}-x_{4})^{2}+(y_{6}-y_{4})^{2}}+\sqrt{(x_{7}-x_{5})^{2}+(y_{7}-y_{5})^{2}}\\
 & \qquad+x_{5}-x_{4}+(x_{7}-x_{6})\sqrt{1+\frac{(n+1)^{2}}{4x_{6}^{4}}}\\
 & =\frac{1}{2}\left(\sqrt{\frac{17(n+1)^{2}}{\left(\sqrt{m^{2}+2(n+1)}+m\right)^{2}}}+\sqrt{\frac{13(n+1)^{2}}{\left(\sqrt{m^{2}+3(n+1)}+m\right)^{2}}}+m\right)\\
 & \qquad+\frac{\lambda(m,n)}{4}\sqrt{\frac{64(n+1)^{2}}{\left(\sqrt{m^{2}+2(n+1)}+m\right)^{4}}+1},
\end{align*}
where
\[
\lambda(m,n):=2\sqrt{m^{2}+3(n+1)}-\sqrt{m^{2}+2(n+1)}+m.
\]
Note that
\begin{align*}
 & 2\sqrt{m^{2}+3(n+1)}-\sqrt{m^{2}+2(n+1)}\\
 & =2\left(\sqrt{m^{2}+3(n+1)}-m\right)-\left(\sqrt{m^{2}+2(n+1)}-m\right)+m\\
 & =\frac{6(n+1)}{\sqrt{m^{2}+3(n+1)}+m}-\frac{2(n+1)}{\sqrt{m^{2}+2(n+1)}+m}+m\\
 & <\frac{4(n+1)}{\sqrt{m^{2}+2(n+1)}+m}+m\\
 & \leq2\sqrt{2(n+1)}+m.
\end{align*}
With this inequality and \eqref{eq:t} we deduce that
\begin{equation}
\begin{aligned}l(\Omega') & <\frac{\sqrt{17(n+1)}}{2}\sqrt{\frac{t}{\left(\sqrt{1+2t}+1\right)^{2}}}+\frac{\sqrt{13(n+1)}}{2}\sqrt{\frac{t}{\left(\sqrt{1+3t}+1\right)^{2}}}\\
 & \qquad+\frac{\sqrt{2(n+1)}}{2}\sqrt{\frac{64t^{2}}{\left(\sqrt{1+2t}+1\right)^{4}}+1}+m\\
 & <\frac{\sqrt{17(n+1)}}{2}\sqrt{\frac{1}{2}}+\frac{\sqrt{13(n+1)}}{2}\sqrt{\frac{1}{3}}+\frac{\sqrt{2(n+1)}}{2}\sqrt{17}+m\\
 & <5.5\sqrt{n+1}+m.
\end{aligned}
\label{eq:1-8}
\end{equation}
It can be easily deduced that 
\begin{equation}
\begin{aligned}\sup_{P,Q\in\Omega'}|x(P)-x(Q)| & =x_{7}-x_{4}=\frac{1}{4}\sqrt{(2m)^{2}+12(n+1)}\\
 & \leq\frac{\sqrt{3(n+1)}}{2}+\frac{m}{2},
\end{aligned}
\label{eq:1-9}
\end{equation}
where in the last step we have used the simple inequality: 
\[
\sqrt{a^{2}+b^{2}}\leq a+b,\quad a\geq0,b\geq0.
\]

Let $N(\Omega')$ denote the number of lattice points in $\Omega'.$
Using Lemmas \ref{lemma2.2} and \ref{lemma2.3}, we derive that
\begin{align*}
M_{2}^{(m)}(n) & =\left|\frac{N(\Omega')}{2}+\left(M_{2}^{(m)}(n)-\frac{N(\Omega')}{2}\right)\right|\\
 & \geq\frac{N(\Omega')}{2}-\left|M_{2}^{(m)}(n)-\frac{N(\Omega')}{2}\right|\\
 & \geq\frac{A(\Omega')-\left|N(\Omega')-A(\Omega')\right|}{2}-\left|M_{2}^{(m)}(n)-\frac{N(\Omega')}{2}\right|\\
 & >\frac{A(\Omega')}{2}-\frac{l(\Omega')}{2}-\sup_{P,Q\in\Omega'}|x(P)-x(Q)|-1.
\end{align*}
Substituting \eqref{eq:1-7}, \eqref{eq:1-8} and \eqref{eq:1-9}
into the right side of the above inequality we easily obtain the result.
\end{proof}

\section{\label{sec:P} Proofs of Theorems \ref{L2}, \ref{lemma2.8} and
\ref{L1}}

In this section, we first prove Theorem \ref{L1} and then show Theorems
\ref{L2} and \ref{lemma2.8}.

\noindent{\it Proof of Theorem \ref{L1}.}  It is easily seen from
\eqref{eq:1-1} that $X^{(m)}(0)=0.$ We suppose $n\geq1$ below.
For any positive integer $n,$ we set $n=2^{e}\cdot N,$ where $e$
is a non-negative integer and $N$ denotes an odd positive integer.

For convinence, we introduce four sets in $\mathbb{N}^{2}.$ Let
\begin{align}
A_{1} & :=\left\{ (d_{1},d_{2})\in\mathbb{N}^{2}|d_{1}d_{2}=N,d_{2}-2^{e}d_{1}\geq2m+1\:is\:odd\right\} ,\nonumber \\
A_{2} & :=\left\{ (d_{1},d_{2})\in\mathbb{\mathbb{N}}^{2}|d_{1}d_{2}=N,d_{2}-2^{e+1}d_{1}\geq2m+1\:is\:odd\right\} ,\label{eq:1-10}\\
B_{1} & :=\left\{ (d_{1},d_{2})\in\mathbb{\mathbb{N}}^{2}|d_{1}d_{2}=N,2^{e+1}d_{2}-d_{1}\geq2m+1\:is\:odd\right\} ,\label{eq:1-11}\\
B_{2} & :=\left\{ (d_{1},d_{2})\in\mathbb{\mathbb{N}}^{2}|d_{1}d_{2}=N,2^{e}d_{2}-d_{1}\geq2m+1\:is\:odd\right\} .\nonumber 
\end{align}
From \eqref{eq:1-1} we deduce that
\begin{equation}
\begin{aligned}Y^{(m)}(n) & =\sum_{\substack{t(t+1)+2mt+2lt=n\\
t\geq1,l\geq0
}
}(-1)^{t}-\sum_{\substack{t(t+1)+2mt+2lt=2n\\
t\geq1,l\geq0
}
}(-1)^{t}\\
 & =\sum_{\substack{t(t+s)=n\\
t\geq1\\
s\geq2m+1\:is\:odd
}
}(-1)^{t}-\sum_{\substack{t(t+s)=2n\\
t\geq1\\
s\geq2m+1\:is\:odd
}
}(-1)^{t}.
\end{aligned}
\label{eq:1-2}
\end{equation}
Then
\begin{equation}
\begin{aligned}Y^{(m)}(n) & =\sum_{(d_{1},d_{2})\in A_{1}}1-\sum_{(d_{1},d_{2})\in B_{2}}1-\sum_{(d_{1},d_{2})\in A_{2}}1+\sum_{(d_{1},d_{2})\in B_{1}}1\\
 & =\#A_{1}+\#B_{1}-\#A_{2}-\#B_{2},
\end{aligned}
\label{eq:1-3}
\end{equation}
where $\#S$ denotes the number of the elements in a finite set $S.$

If $e=0,$ then 
\[
A_{1}=B_{2}=\phi
\]
and 
\[
A_{2}\subseteq B_{1}.
\]
This implies that 
\[
\#A_{1}=\#B_{2}=0
\]
and
\begin{equation}
\#B_{1}\geq\#A_{2}.\label{eq:3-1}
\end{equation}
Thus, from \eqref{eq:1-3} we obtain
\[
Y^{(m)}(n)=\#B_{1}-\#A_{2}\geq0.
\]

If $e\geq1,$ then
\[
A_{2}\subseteq A_{1},\quad B_{2}\subseteq B_{1}
\]
and so 
\[
\#A_{1}\geq\#A_{2},\quad\#B_{1}\geq\#B_{2}.
\]
This, together with \eqref{eq:1-3}, gives 
\[
Y^{(m)}(n)\geq0.
\]
This completes the proof of Theorem \ref{L1}. \qed

\noindent{\it Proof of Theorem \ref{L2}.} From \eqref{eq:1-13} we
know that $Y^{(m)}(0)=0.$ We now consider the case $n\geq1.$ Let\footnote{It is easy to see that $Z^{(m)}(0)=0.$}
\[
\sum_{n\geq1}Z^{(m)}(n)q^{n}:=-\sum_{n\geq1}\frac{(-1)^{n}q^{\frac{n(n+1)}{2}+mn}}{1-q^{n}}
\]
and let $A_{2}$ and $B_{1}$ be defined as in \eqref{eq:1-10} and
\eqref{eq:1-11} respectively. Then
\begin{equation}
\begin{aligned}Z^{(m)}(n) & =-\sum_{\substack{t(t+1)+2mt+2lt=2n\\
t\geq1,l\geq0
}
}(-1)^{t}\\
 & =-\sum_{\substack{t(t+s)=2n\\
t\geq1\\
s\geq2m+1\:is\:odd
}
}(-1)^{t}\\
 & =\sum_{(d_{1},d_{2})\in B_{1}}1-\sum_{(d_{1},d_{2})\in A_{2}}1\\
 & =\#B_{1}-\#A_{2}\geq0,
\end{aligned}
\label{eq:1-12-1}
\end{equation}
where in the last step we have used the inequality \eqref{eq:3-1}.

When $n$ is odd, since
\[
\sum_{\substack{n\geq0\\
n\;is\;odd
}
}X^{(m)}(n)q^{n}=\frac{1}{1-q^{2}}\sum_{\substack{n\geq1\\
n\;is\;odd
}
}Z^{(m)}(n)q^{n},
\]
we deduce from \eqref{eq:1-12-1} that $X^{(m)}(n)$ is nonnegative.

When $n$ is even, it is easy to see that the coefficient of $q^{n}$
in the $q$-series expansion of 
\[
\sum_{n=1}^{\infty}\frac{(-1)^{n}q^{n(3n+1)+2mn}}{1-q^{2n}}
\]
is equal to 
\[
\sum_{\substack{t(3t+1)+2mt+2lt=n\\
t\geq1,l\geq0
}
}(-1)^{t}.
\]
Since
\begin{align*}
\sum_{\substack{t(3t+1)+2mt+2lt=n\\
t\geq1,l\geq0
}
}(-1)^{t} & =\sum_{\substack{t(3t+s)=n\\
t\geq1\\
s\geq1+2m\:is\:odd
}
}(-1)^{t}\\
 & =\sum_{\substack{d_{1}d_{2}=n\\
d_{1}\geq1\:is\:odd,d_{2}\geq1\\
d_{1}-3d_{2}\geq1+2m
}
}1-\sum_{\substack{d_{1}d_{2}=n\\
d_{1}\geq1\:is\:odd,d_{2}\geq1\\
d_{2}-3d_{1}\geq1+2m
}
}1,
\end{align*}
we know that the coefficient of $q^{n}$ in the $q$-series expansion
of 
\[
\sum_{n=1}^{\infty}\frac{(-1)^{n}q^{n(3n+1)+2mn}}{1-q^{2n}}
\]
is
\[
\sum_{\substack{d_{1}d_{2}=n\\
d_{1}\geq1\:is\:odd,d_{2}\geq1\\
d_{1}-3d_{2}\geq1+2m
}
}1-\sum_{\substack{d_{1}d_{2}=n\\
d_{1}\geq1\:is\:odd,d_{2}\geq1\\
d_{2}-3d_{1}\geq1+2m
}
}1.
\]
Then the coefficient of $q^{n}$ in the $q$-series expansion of 
\[
\frac{1}{1-q^{2}}\sum_{n=1}^{\infty}\frac{(-1)^{n}q^{n(3n+1)+2mn}}{1-q^{2n}}
\]
equals
\[
\sum_{\substack{1\leq k\leq n\\
2|k
}
}\left(\sum_{\substack{d_{1}d_{2}=k\\
d_{1}\geq1\:is\:odd,d_{2}\geq1\\
d_{1}-3d_{2}\geq1+2m
}
}1-\sum_{\substack{d_{1}d_{2}=k\\
d_{1}\geq1\:is\:odd,d_{2}\geq1\\
d_{2}-3d_{1}\geq1+2m
}
}1\right).
\]
It can be deduced from \eqref{eq:1-12-1} that 
\begin{align*}
\sum_{\substack{1\leq k\leq n\\
2|k
}
}Z^{(m)}(k) & =\sum_{\substack{1\leq k\leq n\\
2|k
}
}\sum_{\substack{t(t+s)=2k\\
t\geq1\\
s\geq2m+1\:is\:odd
}
}(-1)^{t+1}\\
 & =\sum_{\substack{1\leq k\leq n\\
2|k
}
}\left(\sum_{\substack{d_{1}d_{2}=2k\\
d_{1}\geq1\:is\:odd,d_{2}\geq1\\
d_{2}-d_{1}\geq2m+1
}
}1-\sum_{\substack{d_{1}d_{2}=2k\\
d_{1}\geq1\:is\:odd,d_{2}\geq1\\
d_{1}-d_{2}\geq2m+1
}
}1\right).
\end{align*}
Then the coefficient of $q^{n}$ in the $q$-series expansion of 
\[
-\frac{1}{1-q^{2}}\sum_{n\geq1}\frac{(-1)^{n}q^{\frac{n(n+1)}{2}+mn}}{1-q^{n}}
\]
is equal to 
\[
\sum_{\substack{1\leq k\leq n\\
2|k
}
}\left(\sum_{\substack{d_{1}d_{2}=2k\\
d_{1}\geq1\:is\:odd,d_{2}\geq1\\
d_{2}-d_{1}\geq2m+1
}
}1-\sum_{\substack{d_{1}d_{2}=2k\\
d_{1}\geq1\:is\:odd,d_{2}\geq1\\
d_{1}-d_{2}\geq2m+1
}
}1\right).
\]

In view of the above, we deduce that 
\begin{align*}
X^{(m)}(n) & =\sum_{\substack{1\leq k\leq n\\
2|k
}
}\left(\sum_{\substack{d_{1}d_{2}=k\\
d_{1}\geq1\:is\:odd,d_{2}\geq1\\
d_{1}-3d_{2}\geq2m+1
}
}1-\sum_{\substack{d_{1}d_{2}=k\\
d_{1}\geq1\:is\:odd,d_{2}\geq1\\
d_{2}-3d_{1}\geq2m+1
}
}1+\sum_{\substack{d_{1}d_{2}=2k\\
d_{1}\geq1\:is\:odd,d_{2}\geq1\\
d_{2}-d_{1}\geq2m+1
}
}1-\sum_{\substack{d_{1}d_{2}=2k\\
d_{1}\geq1\:is\:odd,d_{2}\geq1\\
d_{1}-d_{2}\geq2m+1
}
}1\right)\\
 & =\sum_{\substack{1\leq k\leq n\\
2|k
}
}\left(\sum_{\substack{d_{1}d_{2}=k\\
d_{1}\geq1\:is\:odd,d_{2}\geq1\\
d_{1}-3d_{2}\geq2m+1
}
}1-\sum_{\substack{d_{1}d_{2}=k\\
d_{1}\geq1\:is\:odd,d_{2}\geq1\\
d_{2}-3d_{1}\geq2m+1
}
}1+\sum_{\substack{d_{1}d_{2}=k\\
d_{1}\geq1\:is\:odd,d_{2}\geq1\\
2d_{2}-d_{1}\geq2m+1
}
}1-\sum_{\substack{d_{1}d_{2}=k\\
d_{1}\geq1\:is\:odd,d_{2}\geq1\\
d_{1}-2d_{2}\geq2m+1
}
}1\right)\\
 & =\sum_{\substack{1\leq k\leq n\\
2|k
}
}\left(\sum_{\substack{d_{1}d_{2}=k\\
d_{1}\geq1\:is\:odd,d_{2}\geq1\\
2d_{2}-d_{1}\geq2m+1\\
d_{2}-3d_{1}<2m+1
}
}1-\sum_{\substack{d_{1}d_{2}=k\\
d_{1}\geq1\:is\:odd,d_{2}\geq1\\
d_{1}-2d_{2}\geq2m+1\\
d_{1}-3d_{2}<2m+1
}
}1\right).
\end{align*}
From this we obtain 
\begin{align*}
X^{(m)}(n) & =\sum_{\substack{d_{1}d_{2}\leq\frac{n}{2}\\
d_{1}\geq1\:is\:odd,d_{2}\geq1\\
4d_{2}-d_{1}\geq2m+1\\
2d_{2}-3d_{1}<2m+1
}
}1-\sum_{\substack{d_{1}d_{2}\leq\frac{n}{2}\\
d_{1}\geq1\:is\:odd,d_{2}\geq1\\
d_{1}-4d_{2}\geq2m+1\\
d_{1}-6d_{2}<2m+1
}
}1\\
 & =\sum_{\substack{d_{1}d_{2}<\frac{n+1}{2}\\
d_{1}\geq1\:is\:odd,d_{2}\geq1\\
4d_{2}-d_{1}>2m\\
2d_{2}-3d_{1}<2m
}
}1-\sum_{\substack{d_{1}d_{2}<\frac{n+1}{2}\\
d_{1}\geq1\:is\:odd,d_{2}\geq1\\
d_{1}-4d_{2}>2m\\
d_{1}-6d_{2}<2m
}
}1\\
 & =M_{2}^{(m)}(n)-M_{1}^{(m)}(n),
\end{align*}
where $M_{1}^{(m)}(n)$ and $M_{2}^{(m)}(n)$ are defined in Lemmas
\ref{Lemma2.4} and \ref{Lemma2.5} respectively.

It follows from Lemmas \ref{Lemma2.4} and \ref{Lemma2.5} that 
\[
M_{2}^{(m)}(n)-M_{1}^{(m)}(n)>\frac{\log2}{4}(n+1)-6\sqrt{n+1}-m-2.
\]
Then 
\[
X^{(m)}(n)>\frac{\log2}{4}(n+1)-6\sqrt{n+1}-m-2.
\]
When $n\geq f(m),$ we have 
\[
\frac{\log2}{4}(n+1)-6\sqrt{n+1}-m-2\geq0
\]
and so $X^{(m)}(n)>0.$ This finishes the proof of Theorem \ref{L2}.
\qed

\noindent{\it Proof of Theorem \ref{lemma2.8}.} Let
\[
T(q):=\frac{1}{1-q^{2}}\left(\sum_{n=1}^{9}\frac{(-1)^{n}q^{n(3n+1)+2mn}}{1-q^{2n}}-\sum_{n=1}^{19}\frac{(-1)^{n}q^{\frac{n(n+1)}{2}+mn}}{1-q^{n}}\right).
\]
If $n\leq20m,$ then $X^{(m)}(n)$ is equal to the coefficient of
$q^{n}$ in the $q$-series expansion of $T(q).$

We define that
\begin{align*}
T_{1}(q) & :=\frac{1}{1-q^{2}}\left(\frac{q^{1+m}}{1-q}-\frac{q^{3+2m}}{1-q^{2}}-\frac{q^{4+2m}}{1-q^{2}}-\frac{q^{10+4m}}{1-q^{4}}+\frac{q^{14+4m}}{1-q^{4}}\right)\\
 & \quad+\frac{1}{1-q^{2}}\left(\frac{q^{52+8m}}{1-q^{8}}+\frac{q^{200+16m}}{1-q^{16}}-\frac{q^{136+16m}}{1-q^{16}}-\frac{q^{36+8m}}{1-q^{8}}\right),\\
T_{3}(q) & :=\frac{1}{1-q^{2}}\left(\frac{q^{6+3m}}{1-q^{3}}-\frac{q^{21+6m}}{1-q^{6}}-\frac{q^{30+6m}}{1-q^{6}}-\frac{q^{78+12m}}{1-q^{12}}+\frac{q^{114+12m}}{1-q^{12}}\right),\\
T_{5}(q) & :=\frac{1}{1-q^{2}}\left(\frac{q^{15+5m}}{1-q^{5}}-\frac{q^{55+10m}}{1-q^{10}}-\frac{q^{80+10m}}{1-q^{10}}\right),\\
T_{7}(q) & :=\frac{1}{1-q^{2}}\left(\frac{q^{28+7m}}{1-q^{7}}-\frac{q^{154+14m}}{1-q^{14}}-\frac{q^{105+14m}}{1-q^{14}}\right),\\
T_{9}(q) & :=\frac{1}{1-q^{2}}\left(\frac{q^{45+9m}}{1-q^{9}}-\frac{q^{171+18m}}{1-q^{18}}-\frac{q^{252+18m}}{1-q^{18}}\right),
\end{align*}
and
\begin{align*}
T'(q) & :=\frac{q^{11(11+1)/2+11m}}{(1-q^{11})(1-q^{2})}+\frac{q^{13(13+1)/2+13m}}{(1-q^{13})(1-q^{2})}+\frac{q^{15(15+1)/2+15m}}{(1-q^{15})(1-q^{2})}\\
 & \quad+\frac{q^{17(17+1)/2+17m}}{(1-q^{17})(1-q^{2})}+\frac{q^{19(19+1)/2+19m}}{(1-q^{19})(1-q^{2})}.
\end{align*}
Then 
\begin{align*}
T(q) & =T_{1}(q)+T_{3}(q)+T_{5}(q)+T_{7}(q)+T_{9}(q)+T'(q).
\end{align*}
Note that
\begin{align*}
 & \frac{q^{28+7m}}{1-q^{7}}-\frac{q^{105+14m}(1+q^{49})}{1-q^{14}}\\
 & =\frac{q^{28+7m}}{1-q^{7}}-\frac{q^{105+14m}(1-q^{7}+q^{14}-q^{21}+q^{28}-q^{35}+q^{42})}{1-q^{7}}\\
 & =\frac{q^{28+7m}-q^{105+14m}}{1-q^{7}}+q^{105+14m}(q^{7}+q^{21}+q^{35})
\end{align*}
and
\begin{align*}
 & \frac{q^{45+9m}}{1-q^{9}}-\frac{q^{171+18m}(1+q^{81})}{1-q^{18}}\\
 & =\frac{q^{45+9m}}{1-q^{9}}-\frac{q^{171+18m}}{1-q^{9}}\left(1-q^{9}+q^{18}-q^{27}+q^{36}-q^{45}+q^{54}-q^{63}+q^{72}\right)\\
 & =\frac{q^{45+9m}-q^{171+18m}}{1-q^{9}}+q^{171+18m}\left(q^{9}+q^{27}+q^{45}+q^{63}\right).
\end{align*}
Since 
\[
\frac{x^{a}-x^{b}}{1-x}=\sum_{j=a}^{b-1}x^{j},\quad0\leq a<b,
\]
we have
\[
\frac{x^{a}-x^{b}}{1-x}
\]
has non-negative coefficients. From this we know that the coefficients
of the $q$-series expansion in each of the $q$-series 
\[
\frac{q^{28+7m}-q^{105+14m}}{1-q^{7}}
\]
and 
\[
\frac{q^{45+9m}-q^{171+18m}}{1-q^{9}}
\]
are all non-negative and so each of the $q$-series
\[
\frac{q^{28+7m}}{1-q^{7}}-\frac{q^{105+14m}(1+q^{49})}{1-q^{14}}
\]
and 
\[
\frac{q^{45+9m}}{1-q^{9}}-\frac{q^{171+18m}(1+q^{81})}{1-q^{18}}
\]
has also non-negative coefficients. This implies that the coefficients
of $T_{7}(q)$  and $T_{9}(q)$ are all non-negative.

Let
\[
R_{1}(q):=T_{7}(q)+T_{9}(q)+T'(q).
\]
It is easy to see that $T'(q)$ has non-negative coefficients. Thus
$R_{1}(q)$ has also non-negative coefficients.

Note that 
\[
T(q)=T_{1}(q)+T_{3}(q)+T_{5}(q)+R_{1}(q).
\]
After direct computations we find that
\begin{align*}
T(q) & =\frac{1}{1-q^{2}}\left(\sum_{k=1+m}^{2+2m}q^{k}-q^{10+4m}\right)\\
 & \quad+\frac{1}{1-q^{2}}\left(-q^{36+8m}-q^{44+8m}-q^{136+16m}-q^{152+16m}-q^{168+16m}-q^{184+16m}\right)\\
 & \quad+\frac{1}{1-q^{2}}\left(\sum_{k=2+m}^{6+2m}q^{3k}+q^{24+6m}-q^{78+12m}-q^{90+12m}-q^{102+12m}\right)\\
 & \quad+\frac{1}{1-q^{2}}\left(\sum_{k=3+m}^{10+2m}q^{5k}+q^{60+10m}+q^{70+10m}\right)+R_{1}(q).
\end{align*}
Let
\begin{align*}
R_{2}(q) & :=R_{1}(q)+\frac{1}{1-q^{2}}\left(q^{70+10m}+\sum_{1+m\leq k\leq1+2m}q^{k}\right.\\
 & \qquad\left.+\sum_{{2+m\leq k\leq6+2m\atop k\neq2+2m,4+2m,6+2m}}q^{3k}+\sum_{{3+m\leq k\leq10+2m\atop k\neq2+2m,4+2m,6+2m,8+2m}}q^{5k}\right).
\end{align*}
We re-arrange the terms of $T(q)$ to get
\begin{align*}
T(q) & =\frac{1}{1-q^{2}}\left[\left(q^{2+2m}-q^{10+4m}\right)+\left(q^{6+6m}-q^{36+8m}\right)+\left(q^{12+6m}-q^{44+8m}\right)\right.\\
 & \;+\left(q^{18+6m}-q^{78+12m}\right)+\left(q^{24+6m}-q^{90+12m}\right)+\left(q^{60+10m}-q^{102+12m}\right)\\
 & \;+\left(q^{20+10m}-q^{136+16m}\right)+\left(q^{10+10m}-q^{152+16m}\right)+\left(q^{30+10m}-q^{168+16m}\right)\\
 & \left.\;+\left(q^{40+10m}-q^{184+16m}\right)\right]+R_{2}(q).
\end{align*}
It is easy to see that $R_{2}(q)$ is a $q$-series with non-negative
coefficients. From this we obtain that $T(q)$ has non-negative coefficients
and so $X^{(m)}(n)\geq0$ for $n\leq20m.$ This concludes the proof
of Theorem \ref{lemma2.8}. \qed

\section{\label{sec:Pr} Proof of Theorem \ref{t}}

Recall the following generating functions \cite[eqs.(2.1) and (2.2)]{JK}
for $M_{C1}(m,n)$ and $M_{C5}(m,n):$
\[
\sum_{n=1}^{\infty}M_{C1}(m,n)q^{n}=\frac{1}{(q^{2};q^{2})_{\infty}}\left(\sum_{n=1}^{\infty}\frac{(-1)^{n}q^{n(3n+1)+2|m|n}}{1-q^{2n}}-\sum_{n=1}^{\infty}\frac{(-1)^{n}q^{\frac{n(n+1)}{2}+|m|n}}{1-q^{n}}\right)
\]
and
\[
\sum_{n=1}^{\infty}M_{C5}(m,n)q^{n}=\frac{1}{(q^{2};q^{2})_{\infty}}\left(\sum_{n=1}^{\infty}\frac{(-1)^{n}q^{n(n+1)+2|m|n}}{1-q^{2n}}-\sum_{n=1}^{\infty}\frac{(-1)^{n}q^{\frac{n(n+1)}{2}+|m|n}}{1-q^{n}}\right).
\]
From the above generating functions we know that $M_{C1}(m,n)$ and
$M_{C5}(m,n)$ are both symmetric in $m.$ Namely, 
\[
M_{C1}(-m,n)=M_{C1}(m,n)
\]
and
\[
M_{C5}(-m,n)=M_{C5}(m,n).
\]
Thus, we only need to consider the nonnegativity of $M_{C1}(m,n)$
and $M_{C5}(m,n)$ when $m\geq0,n\geq0.$

We first consider $M_{C5}(m,n)$. Since
\[
\sum_{n=1}^{\infty}M_{C5}(m,n)q^{n}=\frac{1}{(q^{2};q^{2})_{\infty}}\cdot\sum_{n\geq0}Y^{(m)}(n)q^{n},
\]
 we deduce from Theorem \ref{L1} that $M_{C5}(m,n)$ is nonnegative
for $m\geq0,n\geq0.$

We next consider $M_{C1}(m,n).$ Let $f(m)$ be defined as in \eqref{eq:t1-1}.
Note that 
\[
f(m)>20m
\]
for $0\leq m\leq120$ and
\[
f(m)<20m
\]
for $m\geq121.$ In view of Theorems \ref{L2} and \ref{lemma2.8},
to study the nonnegativity of $X^{(m)}(n),$ we only need to consider
its nonnegativity for the case $0\leq m\leq120$ and $f(m)>n>20m.$
With the help of the software \emph{Mathematica} we can verify the
finite cases that $X^{(m)}(n)\geq0$ for $0\leq m\leq120$ and $f(m)>n>20m.$
This implies that $X^{(m)}(n)\geq0$ for all $m\geq0,n\geq0.$ Since
\[
\sum_{n=1}^{\infty}M_{C1}(m,n)q^{n}=\prod_{j\geq2}\frac{1}{1-q^{2j}}\cdot\sum_{n\geq0}X^{(m)}(n)q^{n},
\]
we deduce from the nonnegativity of $X^{(m)}(n)$ that $M_{C1}(m,n)\geq0$
for all $m\geq0,n\geq0.$ This completes the proof of Theorem \ref{t}.
\qed

\section*{Acknowledgement}

This work was partially supported by National Natural Science Foundation
of China (Grant No. 11801451) and the Natural Science Foundation of
Changsha (Grant No. kq2208251).

\end{document}